\documentclass[12pt]{amsart}
\usepackage{a4}
\usepackage{amsthm, amsfonts, amssymb,latexsym}
\usepackage{enumerate, color}
\usepackage[english]{babel}
\usepackage[latin1]{inputenc}

\newtheorem{lemma}{Lemma}[section]

\newtheorem{theorem}{Theorem}[section]
\newtheorem{conjecture}{Conjecture}[section]

\newcommand{\la}{\lambda}

\numberwithin{equation}{section}

\newcommand{\beq}[1]{\begin{equation}\label{#1}}
\newcommand{\eeq}{\end{equation}}

\title[Expanders with superquadratic growth]{Expanders with superquadratic growth}

\author[A. Balog, O. Roche-Newton and D. Zhelezov]{Antal Balog, Oliver Roche-Newton and Dmitry Zhelezov}

\address{A. Balog: Alfr\'{e}d R\'{e}nyi Institute of Mathematics, Hungarian Academy of Sciences, Budapest, Hungary }
\email{balog.antal@renyi.mta.hu }

\address{O. Roche-Newton: 69 Altenberger Stra{\ss}e, Johannes Kepler Universit\"{a}t, Linz, Austria }
\email{o.rochenewton@gmail.com }

\address{D. Zhelezov: Mathematical Sciences, Chalmers University of Technology, G\"oteborg, Sweden }
\email{dzhelezov@gmail.com}

\begin{document}

\begin{abstract}
We will prove several expanders with exponent strictly greater than $2$. For any finite set $A \subset \mathbb R$, we prove the following six-variable expander results
\begin{align*}
|(A-A)(A-A)(A-A)| &\gg \frac{|A|^{2+\frac{1}{8}}}{\log^{\frac{17}{16}}|A|},
\\ \left|\frac{A+A}{A+A}+\frac{A}{A}\right| &\gg \frac{|A|^{2+\frac{2}{17}}}{\log^{\frac{16}{17}}|A|},
\\ \left|\frac{AA+AA}{A+A}\right| &\gg \frac{|A|^{2+\frac{1}{8}}}{\log |A|},
\\ \left|\frac{AA+A}{AA+A}\right| &\gg \frac{|A|^{2+\frac{1}{8}}}{\log |A|}.
\end{align*}
\end{abstract} 
\maketitle
\section{Introduction}
Let $A$ be a finite\footnote{From now on, $A,B,C$ etc. will always be finite sets.} set of real numbers. The \textit{sum set} of $A$ is the set $A+A=\{a+b:a,b \in A \}$ and the product set $AA$ is defined analogously. The Erd\H{o}s-Szemer\'{e}di sum-product conjecture\footnote{In fact, the conjecture was originally stated for all $A \subset \mathbb Z$, but it is also widely believed to be true for all $A \subset \mathbb R$.} states that, for any such $A$ and all $\epsilon>0$ there exists an absolute constant $c_{\epsilon}>0$ such that
$$\max \{ |A+A|,|AA| \} \geq c_\epsilon|A|^{2-\epsilon}.$$
In other words, it is believed that at least one of the sum set and product set will always be close to the maximum possible size $|A|^2$, suggesting that sets with additive structure do not have multiplicative structure, and vice versa.

A familiar variation of the sum-product problem is that of showing that sets defined by a combination of additive and multiplicative operations are large. A classical and beautiful result of this type, due to Ungar \cite{ungar}, is the result that for any finite set $A \subset \mathbb R$
\begin{equation}
\left|\frac{A-A}{A-A} \right| \geq |A|^2-2,
\label{ungar}
\end{equation}
where
$$\frac{A-A}{A-A}= \left \{ \frac{a-b}{c-d} : a,b,c,d \in A, c \neq d \right \}.$$
This notation will be used with flexibility to describe sets formed by a combination of additive and multiplicative operations on different sets. For example, if $A,B$ and $C$ are sets of real numbers, then
$AB+C:=\{ab+c:a\in A, b \in B, c \in C \}$. We use the shorthand $kA$ for the $k$-fold sum set; that is $kA:=\{a_1+a_2+\cdots+a_k:a_1,\dots,a_k \in A\}$. Similarly, the $k$-fold product set is denoted $A^{(k)}$; that is $A^{(k)}:=\{a_1a_2\cdots a_k : a_1,\dots,a_k \in A \}$. 

We refer to sets such as $\frac{A-A}{A-A}$, which are known to be large, as \textit{expanders}. To be more precise, we may specify the number of variables defining the set; for example, we refer to $\frac{A-A}{A-A}$ as a \textit{four variable expander}.

Recent years have seen new lower bounds for expanders. For example, Roche-Newton and Rudnev \cite{rectangles} proved\footnote{Throughout the paper, this standard notation
$\ll,\gg$ and respectively $O(\cdot), \Omega(\cdot)$ is applied to positive quantities in the usual way. Saying $X\gg Y$ or $X=\Omega(Y)$ means that $X\geq cY$, for some absolute constant $c>0$. All logarithms in this paper are base $2$.} that for any $A \subset \mathbb R$
\begin{equation}
|(A-A)(A-A)| \gg \frac {|A|^2}{\log|A|},
\label{Mink}
\end{equation}
and Balog and Roche-Newton \cite{BORN} proved that for any set $A$ of strictly positive real numbers
\begin{equation}
\left|\frac{A+A}{A+A} \right| \geq 2|A|^2-1.
\label{ratiosum}
\end{equation}
Note that equations \eqref{ungar}, \eqref{Mink} and \eqref{ratiosum} are optimal up to constant (and in the case of \eqref{Mink}, logarithmic) factors, as can be seen by taking $A=\{1,2,\dots,N\}$. More generally, any set $A$ with $|A+A| \ll |A|$ is extremal for equations \eqref{ungar}, \eqref{Mink} and \eqref{ratiosum}.

With these results, along with others in \cite{GK}, \cite{TJthesis}, \cite{MORNS} and \cite{ORN}, we have a growing collection of near-optimal expander results with a lower bound $\Omega(|A|^2)$ or $\Omega(|A|^2 / \log |A|)$. All of the near-optimal expanders that are known have at least $3$ variables. The aim of this paper is to move beyond this quadratic threshold and give expander results with relatively few variables and with lower bounds of the form $\Omega(|A|^{2+c})$ for some absolute constant $c>0$.

\subsection{Statement of results}

It was conjectured in \cite{BORN} that for any $A \subset \mathbb R$ and any $\epsilon>0$, $|(A-A)(A-A)(A-A)| \gg |A|^{3-\epsilon}$. In this paper, a small step towards this conjecture is made in the form of the following result.

\begin{theorem}  \label{thm:main1} Let $A \subset \mathbb R$. Then
$$|(A-A)(A-A)(A-A)| \gg \frac{|A|^{2+\frac{1}{8}}}{\log^{17/16}|A|}.$$
\end{theorem}
This result is the first improvement on the bound $|(A-A)(A-A)(A-A)| \gg |A|^2 /\log |A|$ which follows trivially from \eqref{Mink}. The proof uses some beautiful ideas of Shkredov \cite{S}.

The following theorem  gives partial support for the aforementioned conjecture from a slightly different perspective.
\begin{theorem}  \label{thm:kfold} Let $A \subset \mathbb R$. Then for any $\epsilon > 0$ there is an integer $k > 0$ such that
$$|(A-A)^{(k)}| \gg_{\epsilon} |A|^{3 - \epsilon}.$$
\end{theorem}

We also prove the following six variables expanders have superquadratic growth.

\begin{theorem}  \label{thm:main2} Let $A \subset \mathbb R$. Then
$$\left|\frac{A+A}{A+A}+\frac{A}{A}\right| \gg \frac{|A|^{2+2/17}}{\log^{16/17}|A|}.$$
\end{theorem}

\begin{theorem}  \label{thm:main3} Let $A \subset \mathbb R$. Then
$$\left|\frac{AA+AA}{A+A}\right| \gg \frac{|A|^{11/8}|AA|^{3/4}}{\log |A|}.$$
In particular, since $|AA| \geq |A|$,
$$\left|\frac{AA+AA}{A+A}\right| \gg \frac{|A|^{2+\frac{1}{8}}}{\log |A|}.$$
\end{theorem}

\begin{theorem}  \label{thm:main4} Let $A \subset \mathbb R$. Then
$$\left|\frac{AA+A}{AA+A}\right| \gg \frac{|A|^{2+\frac{1}{8}}}{\log |A|}.$$
\end{theorem}

The proofs of these three results make use of the results and ideas of Lund \cite{L}.

In fact, a closer inspection of the proof of Theorem \ref{thm:main4} reveals that we obtain the inequality
$$\left | \left \{ \frac{ab+c}{ad+e} :a,b,c,d,e \in A \right \} \right |\gg \frac{|A|^{2+\frac{1}{8}}}{\log |A|}.$$
Therefore, Theorem \ref{thm:main4} actually gives a superquadratic five variable expander.

\section{Preliminary Results}

For the proof of Theorem \ref{thm:main1} we will require the Ruzsa Triangle Inequality. See Lemma 2.6 in Tao-Vu \cite{TV}.
\begin{lemma} \label{triangle} Let $A,B$ and $C$ be subsets of an abelian group $(G,+)$. Then
$$|A-B||C| \leq |A-C||B-C|.$$
\end{lemma}

A closely related result is the Pl\"{u}nnecke-Ruzsa inequality. A simple proof of the following formulation of the Pl\"{u}nnecke-Ruzsa inequality can be found in \cite{P}.

\begin{lemma} \label{Plun} Let $A$ be a subset of an abelian group $(G,+)$. Then
$$|kA-lA| \leq \frac{|A+A|^{k+l}}{|A|^{k+l-1}}.$$
\end{lemma}

We will also use the following variant, which is Corollary 1.5 in Katz-Shen \cite{KS}. The result was originally stated for subsets of the additive group $\mathbb F_p$, but the proof is valid for any abelian group.

\begin{lemma} \label{KSPlun} Let $X,B_1,\dots,B_k$ be subsets of an abelian group $(G,+)$. Then there exists $X' \subset X$ such that $|X'| \geq |X|/2$ and 
$$|X'+B_1+\cdots+B_k| \ll \frac{|X+B_1||X+B_2|\cdots |X+B_k|}{|X|^{k-1}}.$$
\end{lemma}

We will need various existing results for expanders. The first is due to Garaev and Shen \cite{GS}.
\begin{lemma} \label{GS} Let $X,Y,Z \subset \mathbb R$ and $\alpha \in \mathbb R \setminus \{0\}$. Then
$$|XY||(X+\alpha)Z| \gg |X|^{3/2}|Y|^{1/2}|Z|^{1/2}.$$
In particular,
\begin{equation}
|X(X+\alpha)| \gg |X|^{5/4}
\label{GS1}
\end{equation}
and
\begin{equation}
\max \{|XY|,|(X+\alpha)Y| \} \gg |X|^{3/4}|Y|^{1/2}.
\label{GS2}
\end{equation}
\end{lemma}
Note that Lemma \ref{GS} was originally stated only for $\alpha=1$, but the proof extends without alteration to hold for an arbitrary non-zero real number $\alpha$. A similar and earlier result of Elekes, Nathanson and Ruzsa \cite{ENR} will also be used.
\begin{lemma} \label{ENR} Let $f: \mathbb R \rightarrow \mathbb R$ be a strictly convex or concave function and let $X,Y,Z \subset \mathbb R$. Then
$$|f(X)+Y||X+Z| \gg |X|^{3/2}|Y|^{1/2}|Z|^{1/2}.$$
\end{lemma}

Define
$$R[A]:=\left \{\frac{a-b}{a-c} :a,b,c \in A \right \}.$$
The following result is due to Jones \cite{TJthesis}. An alternative proof can be found in \cite{ORN2}.
\begin{lemma} \label{3var} Let $A \subset \mathbb R$. Then
$$|R[A]| \gg \frac{|A|^2}{\log |A|}.$$
\end{lemma}

Each of the three latter results come from simple applications of the Szemer\'{e}di-Trotter Theorem.

Note that the proof of Lemma \ref{3var} also implies that there exists $a,b \in A$ such that
\begin{equation}
|(A-a)(A-b)| \gg \frac{|A|^2}{\log |A|}.
\label{Mink2}
\end{equation}
See \cite{ORN2} for details. In particular, this gives a shorter proof of inequality \eqref{Mink}, requiring only a simple application of the Szemer\'{e}di-Trotter Theorem. The inequality \eqref{Mink} will also be used in the proof of Theorem \ref{thm:main1}.

An important tool in this paper is the following result of Lund \cite{L}, which gives an improvement on \eqref{ratiosum} unless the ratio set $A/A$ is very large.
\begin{lemma} \label{lund} Let $A \subset \mathbb R$. Then
$$\left |\frac{A+A}{A+A}\right | \gg \frac{|A|^2}{\log |A|} \left( \frac{|A|^2}{|A/A|} \right)^{1/8}.$$
\end{lemma}
In fact, a closer examination of the proof of Lemma \ref{lund} reveals that it can be generalised without making any meaningful changes to give the following statement.
\begin{lemma} \label{lund2} Let $A,B \subset \mathbb R$. Then
$$\left |\frac{A+A}{B+B}\right | \gg \frac{|A||B|}{\log |A|+ \log|B|} \left( \frac{|A||B|}{|A/B|} \right)^{1/8}.$$
\end{lemma}

The proofs of Theorems \ref{thm:main2} and \ref{thm:main3} use Lemma \ref{lund2} as a black box. However, for the proof of Theorem \ref{thm:main4} we need to dissect the methods from \cite{L} in more detail and reconstruct a variant of the argument for our problem. To do this, we will also need the following tools which were used in \cite{L}.  The first is a generalisation of the Szemer\'{e}di-Trotter Theorem to certain well-behaved families of curves. A more general version of this result can be found in Pach-Sharir \cite{PS}.

\begin{lemma} \label{thm:PS}
Let $\mathcal P$ be an arbitrary point set in $\mathbb R^2$. Let $\mathcal L$ be a family of curves in $\mathbb R^2$ such that 
\begin{itemize}
\item any two distinct curves from $\mathcal L$ intersect in at most two points and
\item for any two distinct points $p,q \in \mathcal P$, there exist at most two curves from $\mathcal L$ which pass through both $p$ and $q$.
\end{itemize}
Let $K \geq 2$ be some parameter and define $\mathcal L_K:=\{l \in \mathcal L : |l \cap \mathcal P| \geq K \}$. Then
$$|\mathcal L_K| \ll \frac{|\mathcal P|^2}{K^3}+\frac{|\mathcal P|}{K}.$$
\end{lemma}

We will need the following version of the Lov\'{a}sz Local Lemma. This precise statement is Corollary 5.1.2 in \cite{AS}.
\begin{lemma} \label{thm:lovazs} Let $A_1, A_2,\dots, A_n$ be events in an arbitrary probability
space. Suppose that each event $A_i$ is mutually independent from all but at most $d$
of the events $A_j$ with $j \neq i$. Suppose also that the probability of the event $A_i$
occuring is at most $p$ for all $1 \leq i \leq n$. Finally, suppose that
$$ep(d+1) \leq 1.$$
Then, with positive probability, none of the events $A_1,\dots,A_n$ occur.
\end{lemma}

\section{Proof of Theorems \ref{thm:main1} and \ref{thm:kfold}}
\begin{proof}[Proof of Theorem \ref{thm:main1}]
Write $D=A-A$ and apply Lemma \ref{KSPlun} in the multiplicative setting with $k=2$, $X=DD$ and $B_1=B_2=D$. We obtain a subset $X' \subseteq DD$ such that $|X'| \gg |DD|$ and
\begin{equation}
|X'DD| \ll \frac{|DDD|^2}{|DD|}.
\label{eq1}
\end{equation}
Then apply Lemma \ref{triangle}, again in the multiplicative setting, with $A=B=DD$ and $C=(X')^{-1}$. This bounds the left hand side of \eqref{eq1} from below, giving
\begin{equation}
|DD/DD|^{1/2}|X'|^{1/2} \leq |X'DD| \ll \frac{|DDD|^2}{|DD|}.
\label{eq2}
\end{equation}
Recall the observation of Shkredov \cite{S} that $R[A]-1=-R[A]$. Indeed, for any $a,b,c \in A$
$$\frac{a-b}{a-c}-1=\frac{a-b-(a-c)}{a-c}=-\frac{c-b}{c-a}.$$
Therefore, by Lemmas \ref{GS} and \ref{3var},
$$|DD/DD|\geq |R[A]\cdot R[A]|=|R[A]\cdot (R[A]-1)| \gg |R[A]|^{5/4} \gg \frac{|A|^{5/2}}{\log^{5/4} |A|}.$$
Putting this bound into \eqref{eq2} yields
\begin{equation}
\frac{|A|^{5/4}}{\log^{5/8}|A|}|X'|^{1/2} \ll \frac{|DDD|^2}{|DD|}.
\label{eq3}
\end{equation}
Finally, since $|X'| \gg |DD|\gg \frac{|A|^2}{\log |A|}$ by \eqref{Mink}, it follows that
\begin{equation}
|DDD|^2 \gg \frac{|A|^{5/4}}{\log^{5/8}|A|}|DD|^{3/2} \gg \frac{|A|^{5/4}}{\log^{5/8}|A|}\left (\frac{|A|^2}{\log|A|}\right)^{3/2}=\frac{|A|^{17/4}}{\log^{17/8}|A|}.
\label{eq4}
\end{equation}
and thus
$$|DDD| \gg \frac{|A|^{2+\frac{1}{8}}}{\log^{17/16}|A|}$$
as claimed.
\end{proof}

We now turn to the proof of Theorem \ref{thm:kfold}, which exploits similar ideas to the proof of Theorem \ref{thm:main1}.  

\begin{proof}[Proof of Theorem \ref{thm:kfold}]
Let  $R := R[A]$ and $D = A-A$. Further, define  $X_0 = D/D$ and recursively $X_i$ to be either $X_{i-1}R$ or $X_{i-1}(R-1)$ such that
$$
|X_i| = \max \{ |X_{i-1}R|,  |X_{i-1}(R-1)|\}.
$$
We are going to prove by induction on $k$ that
$$
	|X_k| \gg_k \frac{|A|^{3-\frac{1}{2^k}}}{\log^{\frac{3}{2}}|A|}.
$$
Indeed, the base case $k=0$ follows from \eqref{ungar}. Now, let $k \geq 1$. Then applying inequality \eqref{GS2} in Lemma \ref{GS}, Lemma \ref{3var} and the inductive hypothesis
$$
	|X_{k+1}| \gg |X_k|^{1/2}|R|^{3/4} \gg_k \left(\frac{|A|^{3-\frac{1}{2^k}}}{\log^{\frac{3}{2}}|A|}\right)^{1/2} \left(\frac{|A|^2}{\log|A|} \right )^{3/4} =\frac{|A|^{3-\frac{1}{2^{k+1}}}}{\log^{\frac{3}{2}}|A|}.
$$

Now fix $\epsilon > 0$ and choose $k$ sufficiently large so that $\frac{1}{2^k} < \epsilon$.
It was already noted earlier, $R \subseteq D/D$ and $R-1 \subseteq -D/D$, and so 
$$
	|A|^{3-\epsilon} \leq \frac{|A|^{3-\frac{1}{2^k}}}{\log^{\frac{3}{2}}|A|}  \ll_k |X_k| \leq \left| \frac{D^{(k+1)}}{D^{(k+1)}} \right|.
$$
Applying Lemma \ref{triangle} multiplicatively with $A = B = D^{(k+1)}$ and $C=1/D^{(k+1)}$ we obtain that
$$
	|D^{(k+1)}||A|^{3-\epsilon} \ll_\epsilon 	|D^{(2k+2)}|^2,
$$
so $|D^{(2k+2)}| \gg_\epsilon |A|^{3-\epsilon}$. Since $k$ depends on $\epsilon$ only, it completes the proof.
\end{proof}

\subsection{Remarks, improvements and conjectures}

An improvement to Lemma \ref{GS} was given in \cite{JORN}, in the form of the bound
$$|A(A+\alpha)|\gg \frac{|A|^{24/19}}{\log^{2/19} |A|}.$$
Inserting this into the previous argument, we obtain the following small improvement:
$$|DDD| \gg \frac{|A|^{2+\frac{5}{38}}}{\log^{\frac{83}{76}}|A|}.$$
Furthermore, a small modification of the previous arguments can also give the bound
$$|DD/D| \gg \frac{|A|^{2+\frac{5}{38}}}{\log^{\frac{83}{76}}|A|}.$$

In the spirit of Theorem \ref{thm:kfold}, it is reasonable to conjecture the following.
\begin{conjecture} \label{conj:iterateddifference}
For any $l > 0$ there exists $k > 0$ such that 
$$|(A-A)^{(k)}| \gg_{k, l} |A|^{l}$$
uniformly for all sets $A \subset \mathbb{R}$.
\end{conjecture}

 Even the case $l = 3$ is of interest as it is seemingly beyond the limit of the methods of the present paper. An alternative form of Conjecture \ref{conj:iterateddifference} is as follows.
\begin{conjecture} \label{conj:iterateddifference2}
For any $\epsilon > 0$ there exists $\delta > 0$ such that for any real set $X$ with
$$|XX| \leq |X|^{1+\delta}$$ 
the following holds: if $A \subset \mathbb{R}$ is such that 
$$
A - A \subset X,
$$
then
$$
	|A| \ll_\delta |X|^{\epsilon}.
$$
\end{conjecture}
For comparison with Conjecture \ref{conj:iterateddifference}, we note that a similar sum-product estimate with many variables was proven in \cite{BORN}, in the form of the inequality
$$|4^{k-1}A^{(k)}| \gg |A|^k.$$
We also note that Corollary 4 in \cite{SZH} verifies Conjecture \ref{conj:iterateddifference2} for any $\epsilon > 1/2 - c$,  where $c > 0$ is some unspecified (but effectively computable) absolute constant.

It is not hard to see that Conjecture \ref{conj:iterateddifference2} is indeed equivalent to Conjecture \ref{conj:iterateddifference}. Assume that Conjecture \ref{conj:iterateddifference} is true and fix $\epsilon > 0$. Next, take $l = \lfloor 1/\epsilon \rfloor + 3$. Assuming that  Conjecture \ref{conj:iterateddifference} holds, there is $k(\epsilon)$ such that
\begin{equation} \label{eq:differencegrowth}
|(A-A)^{(k)}| \gg_{k, l} |A|^{l} 
\end{equation}
holds for real sets $A$. 

Now, in order to deduce Conjecture \ref{conj:iterateddifference2}, take $\delta = \epsilon/10k$ and assume that there are sets $X, A$ such that $|XX| \leq |X|^{1+\delta}$ and $A-A \subset X$ .
If we now also assume for contradiction that $|A| \geq |X|^{\epsilon}$, then by the Pl\"unnecke-Ruzsa inequality (\ref{Plun})
$$
|(A-A)^{(k)}| \leq |X^{(k)}| \leq |X|^{1+\delta k} \leq |A|^{\frac{1 + \delta k}{\epsilon}} \leq |A|^{l-1},
$$
which contradicts (\ref{eq:differencegrowth}) if $|A|$ is large enough (depending on $\epsilon$), which we can safely assume.

Now let us assume that Conjecture \ref{conj:iterateddifference2} holds true. Let $l > 0$ be fixed and $\epsilon = \frac{1}{l+1}$. Let $A$ be an arbitrary real set. Consider the set
$X_0 = (A-A)$ and define recursively 
$$
	X_{i+1} = X_iX_i.
$$
Note that by construction
$$
    X_i = (A-A)^{(2^{i})}.
$$
Let $c$ be an arbitrary non-zero element in $A-A$. Observe that
$$
   c^{2^i-1} \cdot A - c^{2^i-1} \cdot A = c^{2^i-1} \cdot (A-A) \subset (A-A)^{(2^{i})} = X_i,
$$
and so
$A_i - A_i \subset X_{i} $  where $A_i := c^{2^i-1} \cdot A$. Thus, we are in position to apply the assumption that  Conjecture \ref{conj:iterateddifference2} holds true. In particular, there is $\delta(\epsilon) > 0$ such that  $|A| \ll_\delta |X|^\epsilon$ whenever $A-A \subset X$ and $|XX| \leq |X|^{1+\delta}$. 

Now consider $X_i$ for $i = 1, \ldots, \lfloor l/\delta \rfloor + 1 := j$. For each $i$, if $|X_{i+1}| \leq |X_i|^{1+\delta}$ it follows from Conjecture \ref{conj:iterateddifference2} that $|A| = |A_{i}| \ll_\delta |X_i|^{\epsilon}$, so 
$$
|(A-A)^{(2^i)}| = |X_i| \gg_\delta |A|^{1/\epsilon} \geq |A|^l
$$ 
and we are done. Otherwise, if for each $1 \leq i \leq j$ holds $|X_{i+1}| \geq |X_i|^{1+\delta}$, one has
$$
|(A-A)^{(2^j)}| = |X_j| \geq |X_0|^{1+ j\delta} \geq |A|^l.
$$
Thus, Conjecture \ref{conj:iterateddifference} holds uniformly in $A$ with 
$$
	k(l) := 2^j =2^{ \lfloor l/\delta(l) \rfloor + 1 }.
$$

For a further support, let us remark that Conjecture \ref{conj:iterateddifference2} holds true if one replaces the condition $|XX| \leq |X|^{1+\delta}$ with the more restrictive one $|XX| \leq K|X|$ where $K >0$ is an arbitrary but fixed absolute constant. In this setting Conjecture \ref{conj:iterateddifference2} can be proved by combining the Freiman Theorem and the Subspace Theorem and then applying almost verbatim the arguments of \cite{ORNZ}. We leave the details to the interested reader.


\section{Proofs of Theorems \ref{thm:main2} and \ref{thm:main3}} 

\subsection{Proof of Theorem \ref{thm:main2}}

We will first prove the following lemma.

\begin{lemma} \label{prelim} Let $A \subset \mathbb R$. Then
$$ \left | \frac{A+A}{A+A} + \frac{A}{A} \right | \gg  \frac{|A|^{54/32}|A/A|^{13/32}}{\log^{3/4}|A|}.$$
\end{lemma}

\begin{proof}
Apply Lemma \ref{ENR} with $f(x)=1/x$, $X=(A+A)/(A+A)$ and $Y=Z=A/A$. Note that $f(X)=X$ and so
$$ \left | \frac{A+A}{A+A} + \frac{A}{A} \right | \gg \left | \frac{A+A}{A+A} \right | ^{3/4} |A/A|^{1/2}.$$
Then applying Lemma \ref{lund}, it follows that
$$ \left | \frac{A+A}{A+A} + \frac{A}{A} \right | \gg \frac{|A|^{3/2}}{\log^{3/4}|A|}\left( \frac{|A|^2}{|A/A|} \right)^{\frac{3}{32}} |A/A|^{1/2}=\frac{|A|^{54/32}|A/A|^{13/32}}{\log^{3/4}|A|}. $$
\end{proof}

This immediately implies that 
$$ \left | \frac{A+A}{A+A} + \frac{A}{A} \right | \gg |A|^{2+\frac{3}{32}-\epsilon}.$$
However, by optimising between Lemma \ref{prelim} and Lemma \ref{lund} we can get a slight improvement in the form of Theorem \ref{thm:main2}.

\begin{proof}[Proof of Theorem \ref{thm:main2}]
Let $|A/A|=K|A|$. If $K \geq \frac{|A|^{\frac{1}{17}}}{\log^{\frac{8}{17}}|A|}$ then Lemma \ref{prelim} implies that
$$ \left | \frac{A+A}{A+A} + \frac{A}{A} \right | \gg  \frac{|A|^{67/32}K^{13/32}}{\log^{3/4}|A|} \gg \frac{|A|^{2+2/17}}{\log^{16/17}|A|}.$$
On the other hand, if $K \leq \frac{|A|^{\frac{1}{17}}}{\log^{\frac{8}{17}}|A|}$ then Lemma \ref{lund} implies that
$$ \left | \frac{A+A}{A+A} + \frac{A}{A} \right | \geq  \left | \frac{A+A}{A+A} \right | \gg \frac{|A|^2}{\log |A|}\left ( \frac{|A|}{K} \right )^{1/8} \gg \frac{|A|^{2+2/17}}{\log^{16/17}|A|}.$$
\end{proof}

\subsection{Proof of Theorem \ref{thm:main3}}

Apply Lemma \ref{lund2} with $B=AA$. This yields
$$\left |\frac{AA+AA}{A+A}\right | \gg \frac{|A||AA|}{\log |A|} \left( \frac{|A||AA|}{|A/AA|} \right)^{1/8}.$$
By applying Lemma \ref{Plun} in the multiplicative setting, we have
$$|AA/A| \leq  \frac {|AA|^3}{|A|^2}$$
and so
$$\left |\frac{AA+AA}{A+A}\right | \gg \frac{|A||AA|}{\log |A|} \left( \frac{|A||AA|}{|A/AA|} \right)^{1/8} \geq \frac{|A||AA|}{\log |A|} \left( \frac{|A|^3}{|AA|^2} \right)^{1/8}=\frac{|A|^{11/8}|AA|^{3/4}}{\log |A|}$$
as required.

\section{Proof of Theorem \ref{thm:main4}}

Consider the point set $A\times A$ in the plane. Without loss of generality, we may assume that $A$ consists of strictly positive reals, and so this point set lies exclusively in the positive quadrant. We also assume that $|A| \geq C$ for some sufficiently large absolute constant $C$. For smaller sets, the theorem holds by adjusting the implied multiplicative constant accordingly.

For $\la \in A/A$, let $\mathcal A_{\la}$ denote the set of points from $A \times A$ on the line through the origin with slope $\la$ and let $A_{\la}$ denote the projection of this set onto the horizontal axis. That is,
$$\mathcal A_{\la}:=\{(x,y) \in A \times A:y=\la x\},\,\,\,\,\,\,\,\,A_{\la}:=\{x:(x,y) \in \mathcal A_{\la}\}.$$
Note that $|\mathcal A_{\la}|=|A_{\la}|$ and
$$\sum_{\la} |A_{\la}|=|A|^2.$$

We begin by dyadically decomposing this sum and applying the pigeonhole principle in order to find a large subset of $A\times A$ consisting of points which lie on lines of similar richness. Note that
$$\sum_{\la:|A_{\la}| \leq \frac{|A|^2}{2|A/A|}} |A_{\la}| \leq \frac{|A|^2}{2},$$ 
and so
$$\sum_{\la:|A_{\la}| \geq \frac{|A|^2}{2|A/A|}} |A_{\la}| \geq \frac{|A|^2}{2}.$$
Dyadically decompose the sum to get
$$\sum_{j \geq 1}^{\lceil \log |A| \rceil} \sum_{\la: 2^{j-1}\frac{|A|^2}{2|A/A|} \leq |A_{\la}| < 2^j\frac{|A|^2}{2|A/A|}}|A_{\la}| \geq \frac{|A|^2}{2}.$$
Therefore, there exists some $\tau \geq \frac{|A|^2}{2|A/A|}$ such that
\begin{equation}
\tau|S_{\tau}| \gg \sum_{\la \in S_{\tau}} |A_{\la}| \gg \frac{|A|^2}{\log|A|},
\label{Sbound}
\end{equation}
where $S_{\tau}:=\{ \la : \tau \leq |A_{\la}| <2\tau \}$. Using the trivial bound $\tau \leq |A|$, it also follows that
\begin{equation}
|S_{\tau}| \gg \frac{|A|}{\log|A|}.
\label{Sbound2}
\end{equation}

For a point $p=(x,y)$ in the plane with $x \neq 0$, let $r(p):=y/x$ denote the slope of the line through the origin and $p$. For a point set $P \subseteq \mathbb R^2$ let $r(P):=\{r(p) : p \in P \}$. The aim is to prove that 
\begin{equation}
|r((AA+A) \times (AA+A))|=|r((A \times A ) + (AA \times AA))|\gg \frac{|A|^{2+\frac{1}{8}}}{\log |A|}.
\label{aim}
\end{equation}
Since $r((AA+A) \times (AA+A))= \frac{AA+A}{AA+A}$, inequality \eqref{aim} implies the theorem.

Write $S_{\tau}=\{ \la_1, \la_2,\dots,\la_{|S_{\tau}|} \}$ with $\la_1< \la_2 < \dots < \la_{|S_{\tau}|}$ and similarly write $A=\{x_1,\dots,x_{|A|} \}$ with $x_1< x_2 < \dots < x_{|A|}$ . For each slope $\la_i$, arbitrarily fix an element $\alpha_i \in A_{\la_i}$. Note that, for any $1 \leq i \leq |S_{\tau}|-1$,
\begin{align*}\lambda_i< r((\alpha_i, \la_i \alpha_i)+(\alpha_{i+1}x_1, \la_{i+1} \alpha_{i+1}x_1))&< r((\alpha_i, \la_i \alpha_i)+(\alpha_{i+1}x_2, \la_{i+1} \alpha_{i+1}x_2))
\\& < \dots 
\\& <r((\alpha_i, \la_i \alpha_i)+(\alpha_{i+1}x_{|A|}, \la_{i+1} \alpha_{i+1}x_{|A|}))<\la_{i+1}.
\end{align*}
Since $\la_i \alpha_i$ and $\la_{i+1} \alpha_{i+1}$ are elements of $A$, this gives $|A|$ distinct elements of $R((AA+A) \times (AA+A))$ in the interval $(\la_i,\la_{i+1})$. Summing over all $i$, it follows that
\begin{equation}
 |r((AA+A) \times (AA+A))| \geq \sum_{i=1}^{|S_{\tau}|-1} |A|   = |A|(|S_{\tau}|-1) \gg |A||S_{\tau}|.
\label{basic}
\end{equation}
If $|S_{\tau}| \geq \frac{c|A|^{9/8}}{\log|A|}$ for any absolute constant $c>0$ then we are done. Therefore, we may assume for the remainder of the proof that this is not the case. In particular, by \eqref{Sbound}, we may assume that
\begin{equation}
\tau \geq C |A|^{7/8}
\label{taubound}
\end{equation}
holds for any absolute constant $C$.

Next, the basic lower bound \eqref{basic} will be enhanced by looking at larger clusters of lines, a technique introduced by Konyagin and Shkredov \cite{KS} and utilised again by Lund \cite{L}. We will largely adopt the notation from \cite{L}.

Let $2 \leq M \leq \frac{|S_{\tau}|}{2}$ be an integer parameter, to be determined later. We partition $S_{\tau}$ into clusters of size $2M$, with each cluster split into two subclusters of size $M$, as follows. For each $1\leq t \leq \left \lfloor \frac{|S_{\tau}|}{2M} \right \rfloor$, let 
\begin{align*}
f_t&=2M(t-1)
\\ T_t & = \{ \la_{f_t+1}, \la_{f_t+2}, \dots , \la_{f_t+M} \}
\\ U_t & = \{ \la_{f_t+M+1}, \la_{f_tM++2}, \dots , \la_{f_t+2M} \}.
\end{align*}
For the remainder of the proof we consider the first cluster with $t=1$, but the same arguments work for any $1\leq t \leq \left \lfloor \frac{|S_{\tau}|}{2M} \right \rfloor $. We simplify the notation by writing $T_1=T$ and $U_1=U$.

Let $1 \leq i,k \leq M$ and $M+1 \leq j,l \leq 2M$ with at least one of $i \neq k$ or $j \neq l$ holding. For $a_i \in A_{\la_i}$ and $a_k \in A_{\la_k}$. Define
$$\mathcal E (a_i,j,a_k,l)=| \{ (x,y) \in A \times A : r((a_i,\la_i a_i)+(\alpha_j x,\la_j \alpha_j x))=r((a_k,\la_k a_k)+(\alpha_l y,\la_l \alpha_l y))   |.$$


\begin{lemma} \label{inc} Let $i,j,k,l$ satisfy the above conditions and let $K \geq 2$. Then there are $O(|A|^4/K^3+|A|^2/K)$ pairs $(a_i,a_k) \in  A_{\la_i} \times  A_{\la_k}$ such that
$$\mathcal E (a_i,j,a_k,l) \geq K.$$
\end{lemma}

\begin{proof} We essentially copy the proof of Lemma 2 in \cite{L}, and so some details are omitted. Let $l_{a,b}$ be the curve with equation
$$(\la_i a +\la_j \alpha_j x)(b+ \alpha_l y) = (\la_k b +\la_l \alpha_l y)(a+ \alpha_j x).$$
Let $\mathcal L$ be the set of curves
$$ \mathcal L=\{l_{a,b} : a \in A_{\la_i}, b \in A_{\la_k} \}$$
and let $\mathcal P=A \times A$. Note that $(x,y) \in l_{a_i,a_k}$ if and only if
$$ r((a_i,\la_i a_i)+(\alpha_j x,\la_j \alpha_j x))=r((a_k,\la_k a_k)+(\alpha_l y,\la_l \alpha_l y)). $$
Hence $\mathcal E (a_i,j,a_k,l) \geq K$ if and only if $|l_{a_i,a_k} \cap \mathcal P| \geq K$.

We can verify that the set of curves $\mathcal L$ satisfies the conditions of Lemma \ref{thm:PS}.  One can copy this verbatim from the corresponding part of of the proof of Lemma 2 in \cite{L}. Therefore, there are most
$$O\left ( \frac{|\mathcal P|^2}{K^3} + \frac{|\mathcal P|}{K} \right )= O\left ( \frac{|A|^4}{K^3} + \frac{|A|^2}{K} \right )$$
curves $l \in \mathcal L$ such that $|l \cap \mathcal P| \geq K$. The lemma follows.
\end{proof}

Now, for each $(i,j)$ such that $1 \leq i \leq M$ and $M+1 \leq j \leq 2M$ choose an element $a_{ij} \in A_{\la_i}$ uniformly at random. Then, for any $1\leq i,k \leq M$ and $M+1 \leq j,l \leq 2M$, define $X(i,j,k,l)$ to be the event that
$$\mathcal E (a_{ij},j,a_{kl},l) \geq B,$$
where $B$ is a parameter to be specified later. By Lemma \ref{inc}, the probability that the event $X(i,j,k,l)$ occurs is at most
$$\frac{C}{\tau^2}\left( \frac{|A|^4}{B^3}+\frac{|A|^2}{B} \right ),$$
where $C>0$ is an absolute constant.

Furthermore, note that the event $X(i,j,k,l)$ is independent of the event $X(i',j',k',l')$ unless $(i,j)=(i',j')$ or $(k,l)=(k',l')$. Therefore, the event $X(i,j,k,l)$ is independent of all but at most of $2M^2$ of the other events $X(i',j',k',l')$. With this information, we can apply Lemma \ref{thm:lovazs} with
$$n=M^4-M^2 ,\,\,\,\,\,\,\,\,\, d=2M^2,\,\,\,\,\,\,\,\,\,p=\frac{C}{\tau^2}\left( \frac{|A|^4}{B^3}+\frac{|A|^2}{B} \right ).$$
It follows that there is a positive probability that none of the the events $X(i,j,k,l)$ occur, provided that
\begin{equation}
\frac{eC}{\tau^2}\left( \frac{|A|^4}{B^3}+\frac{|A|^2}{B} \right )(2M^2+1) \leq 1.
\label{probcond}
\end{equation}
The validity of \eqref{probcond} is dependent on our subsequent choice of the value of $B$. For now we proceed under the assumption that this condition is satisfied.

Let
$$Q= \bigcup_{1 \leq i \leq M, M+1 \leq j \leq 2M} \{(a_{ij},\la_i a_{ij})+(\alpha_j a, \la_j \alpha_j a) :a \in A \}.$$
Crucially,
\begin{equation}
r(Q) \geq M^2|A| - \sum_{1\leq i,k \leq M, M+1 \leq j,l \leq 2M : \{i,j\} \neq \{k,l\}} \mathcal E (a_{ij},j,a_{kl},k).
\label{incex}
\end{equation}
In \eqref{incex}, the first term is obtained by counting the $|A|$ slopes in $Q$ coming from all pairs of lines in $U \times T$. The second error term covers the overcounting of slopes that are counted more than once in the first term. 

Since $\mathcal E (a_{ij},j,a_{kl},k) \leq B$ for all quadruples $(i,j,k,l)$ satisfying the aforementioned conditions, it follows that
\begin{equation}
r(Q) \geq M^2|A| - M^4 B.
\label{incex2}
\end{equation}
Choosing $B=\frac{|A|}{2M^2}$, it follows that
\begin{equation}
r(Q) \geq \frac{M^2|A|}{2}.
\label{incex2}
\end{equation}
This choice of $B$ is valid as long as
\begin{equation}
\frac{eC}{\tau^2}( 8M^6|A|+2M^2|A| )(2M^2+1) \leq 1.
\label{probcond2}
\end{equation}
This will certainly hold if
$$\frac{30eC}{\tau^2}M^8|A| \leq 1$$
and so we choose
$$M= \left \lfloor \left ( \frac{\tau^2}{30eC|A|}\right )^{1/8} \right \rfloor.$$
In particular, by \eqref{taubound} we have $M \geq 2$ and so
\begin{equation}
M \gg \frac{\tau^{1/4}}{|A|^{1/8}}.
\label{Mbound}
\end{equation}
It is also true that $M \leq \frac{|S_{\tau}|}{2}$. This is true for all sufficiently large $A$ since
$$|S_{\tau}| \geq \frac{c|A|}{\log|A|} \geq |A|^{1/8} \geq 2M.$$
Therefore
\begin{equation}
\left \lfloor \frac{|S_{\tau}|}{2M} \right \rfloor \gg \frac {|S_{\tau}|}{M}.
\label{intpart}
\end{equation}

Next, note that $r(Q)$ is a subset of the interval $(\la_1,\la_{2M})$. We can repeat this argument for the next cluster to find at least $M^2|A|/2$ elements of $r((AA+ A)\times (AA+A))$ in the interval $(\la_{2M+1},\la_{4M})$ and then so on for each of the $\left \lfloor \frac{|S_{\tau}|}{2M} \right \rfloor$ clusters of size $2M$. It then follows from \eqref{intpart} and \eqref{Mbound} that
\begin{align*}
 \left | \frac{AA+A}{AA+A} \right |&=|r((AA+ A)\times (AA+A))| 
\\&\geq \sum_{j=1}^{\left \lfloor \frac{|S_{\tau}|}{2M} \right \rfloor} \frac{M^2|A|}{2}
\\& \gg |S_{\tau}|M|A|
\\&\gg (|S_{\tau}|\tau)^{1/4}|A|^{7/8}|S_{\tau}|^{3/4}.
\end{align*}
Applying \eqref{Sbound} and \eqref{Sbound2}, we conclude that
$$ \left | \frac{AA+A}{AA+A} \right | \gg \frac{|A|^{2+\frac{1}{8}}}{\log |A|}$$
as required.

\section*{Acknowledgement}

The research of Antal Balog was supported by the Hungarian National Science Foundation Grants NK104183 and
K109789. Oliver Roche-Newton was supported by the Austrian Science Fund FWF Project F5511-N26,
which is part of the Special Research Program "Quasi-Monte Carlo Methods: Theory and Applications".


\end{document}